\documentclass[10pt,psamsfonts]{amsart}
\usepackage{amsmath}
\usepackage{amsthm}
\usepackage{amssymb}
\usepackage{amscd}
\usepackage{amsfonts}
\usepackage{amsbsy}
\usepackage{graphicx}
\usepackage[dvips]{psfrag}
\usepackage{array}
\usepackage{color}
\usepackage{epsfig}
\usepackage{url}
\usepackage{overpic}
\usepackage{tikz-cd}
\usepackage{enumitem}
\usepackage{hyperref}
\usepackage{soul}
\usepackage{listings}

\newtheorem {theorem} {Theorem}

\newtheorem {lemma}  [theorem]{Lemma}

\newtheorem {remark} [theorem]{Remark}

\newtheorem {mtheorem} {Theorem}

\newtheorem {mcorollary} {Corollary}

\usepackage{tabularx}
\newcommand{\R}{\ensuremath{\mathbb{R}}}

\newcommand{\CO}{\ensuremath{\mathcal{O}}}

\newcommand{\T}{\theta}
\newcommand{\vf}{\varphi}

\newcommand{\bg}{{\bf g}}
\newcommand{\bu}{{\bf u}}
\newcommand{\f}{{\bf f}}

\def\p{\partial}
\def\e{\varepsilon}

\usepackage{mathpazo}

\title[On the higher order stroboscopic averaged functions]
{Higher order stroboscopic averaged functions:\\ a general relationship with Melnikov functions}

\author[D.D. Novaes]
{Douglas D. Novaes$^{1}$}

\address{$^1$ Departamento de Matem\'{a}tica - Instituto de Matem\'{a}tica, Estat\'{i}stica e Computa\c{c}\~{a}o Cient\'{i}fica (IMECC) - Universidade
Estadual de Campinas (UNICAMP),  Rua S\'{e}rgio Buarque de Holanda, 651, Cidade Universit\'{a}ria Zeferino Vaz, 13083-859, Campinas, SP,
Brazil} \email{ddnovaes@unicamp.br}



\begin{document}

\maketitle

\begin{abstract}
In the research literature, one can find distinct notions for higher order averaged functions of regularly perturbed non-autonomous $T$- periodic differential equations of the kind  $x'=\e F(t,x,\e)$.  By one hand, the classical (stroboscopic) averaging method provides asymptotic estimates for its solutions in terms of some uniquely defined functions $\bg_i$'s,  called averaged functions, which are obtained through near-identity stroboscopic transformations and by solving homological equations. On the other hand, a Melnikov procedure is employed to obtain bifurcation functions $\f_i$'s which controls in some sense the existence of isolated $T$-periodic solutions of the differential equation above. In the research literature, the bifurcation functions $\f_i$'s are sometimes likewise called averaged functions, nevertheless, they also receive the name of Poincar\'{e}-Pontryagin-Melnikov functions or just Melnikov functions. While it is known that $\f_1=T \bg_1,$ a general relationship between $\bg_i$ and $\f_i$ is not known so far for $i\geq 2.$ Here,  such a general relationship between these two distinct notions of averaged functions is provided, which allows the computation of the stroboscopic averaged functions of any order avoiding the necessity of dealing with near-identity transformations and homological equations. In addition, an Appendix is provided with implemented Mathematica algorithms for computing both higher order averaging functions.
\end{abstract}

\section{Introduction}

This paper is dedicated to investigate the link between two distinct notions of higher order averaged functions of regularly perturbed non-autonomous $T$- periodic differential equations of the kind $x'=\e F(t,x,\e)$.

The first notion comes from the classical averaging method, which provides asymptotic estimates for the solutions of the differential equation $x'=\e F(t,x,\e)$ in terms of some uniquely defined functions, called averaged functions, which are obtained through near-identity stroboscopic transformations and by solving homological equations.

The second notion is provided by the Melnikov method, where the averaged functions are obtained by expanding the time-$T$ map of the differential equation $x'=\e F(t,x,\e)$ around $\e=0$ and control, in some sense, the bifurcation of isolated $T$-periodic solutions. 

In the sequel, these notions will be discussed in detail.

\subsection{The averaging method}
An important and celebrated tool for dealing with nonlinear oscillating systems in the presence of small perturbations is the {\it averaging method}, which has its foundations in the works of Clairaut, Laplace, and Lagrange, in the development of celestial mechanics, and was rigorous formalized by the works of Fatou, Krylov, Bogoliubov, and Mitropolsky  \cite{BM,Bo,Fa,BK} (for a brief historical review, see  \cite[Chapter 6]{MN} and \cite[Appendix A]{SVM}). It is mainly concerned in providing long-time asymptotic estimates for solutions of non-autonomous differential equations given in the following standard form 
 \begin{equation}\label{eq:e1}
	x' = \sum_{i=1}^k\e F_i(t,x)+\e^{k+1}R(t,x,\e).
	\end{equation}
	Here, $F_i\colon\R\times D\to \R^n,$ for $i=1,\ldots,k,$ and $R\colon\R\times D\times[-\e_0,\e_0]\to \R^n$ are assumed to be smooth functions $T-$periodic in the variable $t,$ with $D$ being an open subset of $\R^n$ and $\e_0 > 0$ small. Such asymptotic estimates are given in terms of solutions of an autonomous {\it truncated averaged equation}
	\begin{equation}\label{taq}
	\xi'=\sum_{i=1}^k\e^i \bg_i(\xi),
	\end{equation}
	where $\bg_i:D\rightarrow\R^n,$ for $i\in\{1,\ldots,k\},$ are obtained by the following result:
	\begin{theorem}[{\cite[Lemma 2.9.1]{SVM}}]\label{thm:av1}
	There exists a smooth $T$-periodic near-identity transformation
	\[
	x=U(t,\xi ,\e)=\xi+\sum_{i=1}^k \e^i\, \bu_i(t,\xi ),
	\]
	 satisfying $U(0,\xi,\e)=\xi$, such that the differential equation \eqref{eq:e1} is transformed into
	\begin{equation}\label{fullaveq}
	\xi'=\sum_{i=1}^k\e^i\bg_i(\xi)+\e^{k+1} r_k(t,\xi,\e).
	\end{equation}
	\end{theorem}
	
	The averaging theory states that, for $|\e|\neq0$ sufficiently small, the solutions of the original differential equation \eqref{eq:e1} and the truncated averaged equation \eqref{taq}, starting at the same initial condition, remains $\e^k$-close for a time interval of order $1/\e$ (see \cite[Theorem 2.9.2]{SVM}).
		
	The functions $\bg_i$ and $\bu_i$ can be algorithmically computed by solving homological equations. Section 3.2 of \cite{SVM} is devoted to discuss how is the best way to work with such near-identity transformations based on Lie theory (see also \cite{ellison,perko}). One can see that, in general, $\bg_1$ is the average of $F_1(t,\cdot),$ that is,
	\[
	\bg_1(z)=\dfrac{1}{T}\int_0^T F_1(t,x)dt.
	\]
	It is worth mentioning that the so-called {\it stroboscopic condition} $U(\xi,0,\e)=\xi$ does not have to be assumed in order to get \eqref{fullaveq}. However, in that case, the functions $\bg_i,$ for $i\geq2,$ are not uniquely determined. For the stroboscopic averaging, the uniqueness of each $\bg_i$ is guaranteed and so it is natural to call it by {\it averaged function of order $i$} (or {\it $i$th-order averaged function})  of the differential equation \eqref{eq:e1}. Here, these functions are referred by {\it stroboscopic averaged functions} to indicate that the stroboscopic condition is being assumed.
	
The averaging method has been employed in the investigation of invariant manifolds of differential equations (see, for instance, \cite{hale61}). In particular, it has been extensively used to study periodic solutions of differential equations. One can find results, in the classical research literature, that relate simple zeros of the first-order averaged function $\bg_1$ with isolated $T$-periodic solutions of \eqref{eq:e1} (see, for instance, \cite{Hale,SVM,Ver06}).

\subsection{The Melnikov method}\label{sec:mm}

The mentioned results relating simple zeros of the first-order averaged function $\bg_1$ with isolated $T$-periodic solutions of \eqref{eq:e1} have been generalized in several directions (see, for instance, \cite{buicua2004averaging,CMZ21,LliNovRod17,LliNovTei2014,LlilNovTei15, Novaes2021,NovaesSilva}). In particular, a recursively defined sequence of functions $\f_i:D\rightarrow \R^n,$ $i\in\{1,\ldots,k\},$ was obtained  in \cite{LliNovTei2014}, for which the following result holds:
\begin{theorem}[\cite{LliNovTei2014}]
Denote $\f_0=0.$ Let $\ell\in\{1,\ldots,k\}$ satisfying $\f_0=\cdots\f_{\ell-1}=0$ and $\f_{\ell}\neq0.$ Assume that $z^*\in D$ is a simple zero of $\f_{\ell}.$ Then, for $|\e|\neq0$ sufficiently small, the differential equation \eqref{eq:e1} admits an isolated $T$-periodic solution $\vf(t,\e)$ such that $\vf(0,\e)\to z^*$ as $\e\to 0.$ 
\end{theorem}
The bifurcation functions $\f_i,$ $i\in\{1,\ldots,k\},$ are obtained through a Melnikov procedure, which consists in expanding the time-$T$ map of the differential equation \eqref{eq:e1} around $\e=0$ by using the following result:
\begin{lemma}[\cite{LliNovTei2014,N17}]
\label{lem:fund}Let $x(t,z,\e)$ be the solution of \eqref{eq:e1} satisfying
$x(0,z,\e)=z$. Then,
\[
x(t,z,\e)=z+\sum_{i=1}^{k}\e^{i}\dfrac{y_i(t,z)}{i!}+\CO(\e^{k+1}), 
\]
where 
\begin{equation}\label{yi}
\begin{aligned}
y_1(t,z)=& \int_0^tF_1(s,z)\,ds\,\, \text{ and }\vspace{0.3cm}\\
y_i(t,z)=& \int_0^t\bigg(i!F_i(s,z)+\sum_{j=1}^{i-1}\sum_{m=1}^j\dfrac{i!}{j!}\p_x^m F_{i-j} (s,z)B_{j,m}\big(y_1,\ldots,y_{j-m+1}\big)(s,z)\bigg)ds,
\end{aligned}
\end{equation}
for $i\in\{2,\ldots,k\}.$
\end{lemma}

As usual, for  $p$ and $q$ positive integers, $B_{p,q}$ denotes the  {\it partial Bell polynomials}:
\[
B_{p,q}(x_1,\ldots,x_{p-q+1})=\sum\dfrac{p!}{b_1!\,b_2!\cdots b_{p-q+1}!}\prod_{j=1}^{p-q+1}\left(\dfrac{x_j}{j!}\right)^{b_j}.
\]
The sum above is taken  
over all the tuples of nonnegative integers $(b_1,b_2,\cdots,b_{p-q+1})$ satisfying $b_1+2b_2+\cdots+(p-q+1)b_{p-q+1}=p$ and
$b_1+b_2+\cdots+b_{p-q+1}=q.$ Here, $\p_x^m F_{i-j} (s,z)$ denotes the Frechet's derivative of $F_{i-j}$ with respect to the variable $x$ evaluated at $x=z$, which is a symmetric $m$-multilinear map that is applied to combinations of ``products'' of $m$ vectors in $\R^n,$ in the present case $B_{j,m}\big(y_1,\ldots,y_{j-m+1}\big).$

Accordingly, the functions $\f_i,$ for $i\in\{1,\ldots,k\},$ are defined by
\begin{equation}\label{avfunc}
\f_i(z)=\dfrac{y_i(T,z)}{i!}.
\end{equation}
Notice that $\f_1$ is the average of $F_1(t,\cdot)$ multiplied by a factor $T$, that is, $\f_1=T\bg_1$. Usually, $\f_i$ is likewise called by {\it averaged function of order $i$} (or {\it $i$th-order averaged function}) of the differential equation \eqref{eq:e1}. It is worth mentioning that, in the research literature, the bifurcation functions $\f_i$'s also receive the name of {\it Poincar\'{e}-Pontryagin-Melnikov functions}  or just {\it Melnikov functions}. Such functions can be easily formally computed from \eqref{yi} and \eqref{avfunc}, for instance
\begin{equation}\label{f2}
\begin{aligned}
\f_2(z)=&\int_0^T\bigg(F_2(t,z)+\p_x F_1(t,z)y_1(t,z)\bigg)dt\\
=&\int_0^T\bigg(F_2(t,z)+\p_x F_1(t,z)\int_0^t F_1(s,z)ds \bigg)dt.\\
\end{aligned}
\end{equation}

\subsection{Main goals}
In a first view, the functions $\f_i$ and $\bg_i,$ for $i\geq 2,$ do not hold a clear relationship. Thus, the present study is mainly concerned in establishing a link between them.

In Section \ref{sec:mr}, the main result of this paper, Theorem \ref{main}, provides a general relationship between such distinct notions of higher order averaged functions, which allows the computation of the higher order stroboscopic averaged functions avoiding the necessity of dealing with near-identity transformations and homological equations. In addition, an Appendix is provided with implemented Mathematica algorithms for computing both higher order averaging functions. 

In Section \ref{sec:cons}, some consequences of the main result are presented. First, Corollary \ref{main} states that $\f_i=T\,\bg_i,$ for $i\in\{1,\ldots,\ell\},$ provided that either $\f_1=\cdots=\f_{\ell-1}=0$ or $\bg_1=\cdots=\bg_{\ell-1}=0.$ This gives a relatively simple and computable expression for the first non-vanishing stroboscopic averaged function (see Corollary \ref{c1}).  This last result has been reported in \cite{Han2015} for differential equations coming from planar near-Hamiltonian systems.

\section{Related results in research literature}\label{sec:known}
In this section, some known results in research literature regarding the relationship between Melnikov functions and averaged functions are discussed.

In \cite{Han2015}, the authors have investigated the relationship between averaged functions and Melnikov functions for planar near-Hamiltonian systems 
\[
\dot x=H_y+\e f(x,y,\e),\quad \dot y=-H_x+\e g(x,y,\e), \quad (x,y)\in\R^2,
\]
assuming that the unperturbed system $\dot x=H_y,\quad \dot y=-H_x$ has a continuous family of periodic solutions $L_h$, $h\in J\subset \mathbb{R}$. It is worthy mentioning that this is the natural context where Melnikov theory is applied.  After a change of variables $(x,y)\in\R^2\mapsto(\T,h)\in [0,2\pi)\times J$, the near-Hamiltonian system can be written in the standard form \eqref{eq:e1},
\[
\dfrac{d h}{d\T}=\e F(\T,h,\e), \quad (\T,h)\in [0,2\pi)\times J
\]
(see  \cite[Lemma 2.2]{Han2015}), for which the Melnikov functions $\f_i$'s and the stroboscopic averaged functions $\bg_i$'s can be computed. Then, in \cite[Theorem 3.1]{Han2015}, they showed that  $\f_i=T\,\bg_i,$ for $i\in\{1,\ldots,\ell\},$ provided that $\f_1=\cdots=\f_{\ell-1}=0$.

Although less related with the present study, another interesting paper to be mention is \cite{buica17}, where the author considered planar autonomous differential equations given by
\begin{equation}\label{eq:buica1}
\dot x=X_0(x)+\e X(x,\e),
\end{equation}
for which the unperturbed system $\dot x=X_0(x)$ has a continuous period annulus $\mathcal{P}\subset \R^2$ without equilibria.  A polar-like change of variables is employed in order to write the planar system as the standard form \eqref{eq:e1},
\begin{equation}\label{eq:buica2}
\dfrac{d h}{d\T}=\e F(\T,h,\e), \quad (\T,h)\in [0,2\pi)\times J,
\end{equation}
(see \cite[Propositions 4 and 5]{buica17}). The {\it averaging method for scalar periodic equations}, described in \cite[Section 1]{buica17}, corresponds to the Melnikov method described in Section \ref{sec:mm} of this present paper, where the averaged functions $f_i$'s are given as the coefficients of the expansion of the time-$T$ map of the  differential equation \eqref{eq:buica2} around $\e=0.$  The {\it Melnikov function method for planar systems}, also described in \cite[Section 1]{buica17}, consider a Poincar\'{e} map $P^{\gamma}$ of the autonomous differential equation \eqref{eq:buica1} defined on an analytic transversal section given by $\Sigma=\{\gamma(h):h\in I\}$. Accordingly, the Melnikov functions $M_i$'s are given as the coefficients of the expansion of the Poincar\'{e} map around $\e=0,$ which may depend on both the section $\Sigma$ and its parametrization $\gamma.$  As the conclusion of \cite{buica17}, it was showed that both procedure correspond to the study of some Poincar\'{e} map and that $M_{\ell}=f_{\ell},$ where $\ell$ is the index of the first non-vanishing Melnikov function.

\section{Main result}\label{sec:mr}

The main result of this paper establishes a general relationship between the distinct notions of higher order averaged functions provided by the stroboscopic averaging method in Theorem \ref{thm:av1} and by the Melnikov procedure in Lemma \ref{lem:fund}. 

\begin{mtheorem}\label{c2} For $i\in\{1,\ldots,k\},$ the following recursive relationship between $\bg_i$ and $\f_i$ holds:
\begin{equation}\label{rec:gi}
\begin{aligned}
\bg_1(z)=&\dfrac{1}{T}\f_1(z),\\
\bg_i(z)=&\dfrac{1}{T}\left( \f_i(z)-\sum_{j=1}^{i-1}\sum_{m=1}^j\dfrac{1}{j!}d^m \bg_{i-j}(z) \int_0^T B_{j,m}\big(\tilde y_1,\ldots,\tilde y_{j-m+1}\big)(s,z)ds\right),
\end{aligned}
\end{equation}
where $\tilde y_i(t,z),$ for $i\in\{1,\ldots,k\},$ are polynomial in the variable $t$ recursively defined as follows:
\begin{equation}\label{tildeyi}
\begin{aligned}
\tilde y_1(t,z)=&t\, \bg_1(z)\vspace{0.3cm}\\
\tilde y_i(t,z)=& i!t\,\bg_i(z) +\sum_{j=1}^{i-1}\sum_{m=1}^j\dfrac{i!}{j!}d^m \bg_{i-j}(z) \int_0^t B_{j,m}\big(\tilde y_1,\ldots,\tilde y_{j-m+1}\big)(s,z)ds.
\end{aligned}
\end{equation}
\end{mtheorem}

Theorem \ref{c2} is proven in Section \ref{proof:main}. An Appendix is provided with Mathematica algorithms implementing the recursive formulae \eqref{yi}, \eqref{rec:gi}, and \eqref{tildeyi} for computing both higher order averaging functions.one has

\begin{remark}
By applying the formula above for $i=2$, one has
\[
\bg_2(z)=\dfrac{1}{T}\left(\f_2(z)-\dfrac{1}{2}d\f_1(z) \f_1(z)\right),
\]
where $\f_2$ is explicitly given by \eqref{f2}. Thus,
\[
\bg_2(z)=\dfrac{1}{T}\int_0^T\bigg(F_2(t,z)+\p_x F_1(t,z)\int_0^t\Big(F_1(s,z)-\dfrac{1}{2}\bg_1(z)\Big)ds\bigg)dt,
\]
which coincides with the expression provided by \cite[Section 2.9.1]{SVM}.
\end{remark}

\subsection{Proof of Theorem \ref{c2}}\label{proof:main}
From Theorem \ref{thm:av1}, there exists a $T$-periodic near-identity transformation $x=U(t,\xi ,\e)$, satisfying $U(0,\xi,\e)=\xi$, such that the differential equation \eqref{eq:e1} is transformed into \eqref{fullaveq}.
Let $x(t,z,\e)$ and $\xi(t,z,\e)$ be, respectively, the solutions of \eqref{eq:e1} and \eqref{fullaveq} satisfying $x(0,z,\e)=\xi(0,z,\e)=z.$ From Lemma \ref{lem:fund}, 
\begin{equation}\label{xexp}
x(T,z,\e)=z+\sum_{i=1}^{k}\e^{i}\f_i(z)+\CO(\e^{k+1}),
\end{equation}
where $\f_i,$ for $i\in\{1,\ldots,k\},$ are given by  \eqref{avfunc}, and
\begin{equation}\label{xiexp}
\xi(T,z,\e)=z+\sum_{i=1}^{k}\e^{i}\tilde\f_i(z)+\CO(\e^{k+1}),
\end{equation}
where, for $i\in\{1,\ldots, k\},$
\begin{equation}\label{tildef}
\tilde\f_i(z)=\dfrac{\tilde y_i(T,z)}{i!}
\end{equation}
and the functions $\tilde y_i$'s are obtained recursively from \eqref{yi} as \eqref{tildeyi}:
\[
\begin{aligned}
\tilde y_1(t,z)=&t\, \bg_1(z)\vspace{0.3cm}\\
\tilde y_i(t,z)=& i!t\,\bg_i(z) +\sum_{j=1}^{i-1}\sum_{m=1}^j\dfrac{i!}{j!}d^m \bg_{i-j}(z) \int_0^t B_{j,m}\big(\tilde y_1,\ldots,\tilde y_{j-m+1}\big)(s,z)ds.
\end{aligned}
\]

Now, taking the transformation $x=U(t,\xi ,\e)$ into account, given $z\in D$, there exists $\hat z\in D$ such that $x(t,z,\e)=U(t,\xi(t,\hat z,\e),\e).$ By the stroboscopic condition, $U(0,\xi,\e)=\xi,$
it follows that 
 \[
 z=x(0,z,\e)=U(0,\xi(0,\hat z,\e),\e)=\xi(0,\hat z,\e)=\hat z.
 \]
 In addition, since $U$ is $T$ periodic in the variable $t,$ one also has that
 \[
 x(T,z,\e)=U(T,\xi(T,z,\e),\e)=\xi(T,z,\e).
 \]
 Thus, from \eqref{xexp} and \eqref{xiexp}, one obtains the following relationship
 \begin{equation}\label{relation}
 \f_i=\tilde\f_i\quad \text{for every}\quad i\in\{1,\ldots,k\}.
 \end{equation}
 
 The proof follows by substituting  \eqref{tildef} into the equality \eqref{relation} and, then, isolating $\bg_i(z)$ in the resulting relation by taking \eqref{tildeyi} into account.

\subsection{Some consequences}\label{sec:cons}
Two main consequences of Theorem \ref{c2} are given in the sequel. The first one states that $\f_i=T\,\bg_i,$ for $i\in\{1,\ldots,\ell\},$ provided that either $\f_1=\cdots=\f_{\ell-1}=0$ or $\bg_1=\cdots=\bg_{\ell-1}=0.$ This generalizes the result from \cite{Han2015} discussed in Section \ref{sec:known}.

\begin{mcorollary}\label{main}
Let $\ell\in\{2,\ldots,k\}$. If either $\f_1=\cdots=\f_{\ell-1}=0$ or $\bg_1=\cdots=\bg_{\ell-1}=0,$ then $\f_i=T\,\bg_i$, for $i\in\{1,\ldots,\ell\}.$
\end{mcorollary}
\begin{proof}
 
First, assume that $\bg_1=\cdots=\bg_{\ell-1}=0.$ Then, from \eqref{tildeyi}, $\tilde y_i=0,$ for $i\in\{1,\ldots,\ell-1\},$ and $\tilde y_{\ell}(t,z)=\ell! \, t\, \bg_{\ell}(z).$ Thus, from \eqref{tildef} and \eqref{relation}, it follows that $\f_i=\tilde\f_ i=0,$ for $i\in\{1,\ldots,\ell-1\},$ and $\f_{\ell}=\tilde\f_ {\ell}=T\bg_{\ell}.$

Finally, assume that $\f_1=\cdots=\f_{\ell-1}=0.$ Then, from \eqref{tildef}, \eqref{tildeyi}, and \eqref{relation}, one concludes that $\bg_i,$ for $i\in\{1,\ldots,\ell-1\},$ satisfy the following system of equations
\[
\begin{aligned}
0&=T\,\bg_1(z),\\
0&=T\,\bg_i(z) +\sum_{j=1}^{i-1}\sum_{m=1}^j\dfrac{1}{j!}d^m \bg_{i-j}(z) \int_0^t B_{j,m}\big(\tilde y_1,\ldots,\tilde y_{j-m+1}\big)(s,z)ds.
\end{aligned}
\]
Hence, $\bg_i=0$ (and, then, $\f_{i}=T\bg_{i}$)  for $i\in\{1,\ldots,\ell-1\}.$ Consequently, applying \eqref{tildeyi} for $i=\ell$ and taking \eqref{tildef} and \eqref{relation} into account, if follows that 
\[
\f_{\ell}(z)=\tilde \f_{\ell}(z)=\dfrac{\tilde y_{\ell}(T,z)}{\ell!}=T \bg_{\ell}(z).
\]
\end{proof}

Now, as a direct consequence of Corollary \ref{main}, the first non-vanishing stroboscopic averaged function can be computed in relatively simple way. In particular, the following result holds:
\begin{mcorollary}\label{c1}
Denote $\f_0=0$ and let $\ell\in\{1,\ldots,k\}$ satisfy $\f_1=\cdots\f_{\ell-1}=0.$ Then, there exists a smooth $T$-periodic near-identity transformation $x=U(t,\xi ,\e)$
	 satisfying $U(\xi,0,\e)=\xi$, such that the differential equation \eqref{eq:e1} is transformed into
\[
	\xi'=\e^{\ell}\dfrac{1}{T}\f_{\ell}(\xi)+\e^{\ell+1} r_{\ell}(t,\xi,\e).
\]
\end{mcorollary}

\section*{Appendix A: Algorithms}\label{sec:alg}

\lstset{language=Mathematica}
\lstset{basicstyle={\sffamily\footnotesize},
	numbers=left,
	numberstyle=\tiny\color{gray},
	numbersep=5pt,
	breaklines=true,
	captionpos={t},
	frame={lines},
	rulecolor=\color{black},
	framerule=0.5pt,
	columns=flexible,
	tabsize=2
}

This appendix is devoted to provide implemented Mathematica algorithm, based on recursive formulae  \eqref{yi}, \eqref{rec:gi}, and \eqref{tildeyi}, for computing the higher order Melnikov functions and the higher order stroboscopic averaged functions.

In what follows, $F_i(t,x)$ is denoted by {\sffamily F[i,t,x]}, $y_i(t,x)$ is denoted by  {\sffamily y0[i,t]}, $\tilde y_i(t,x)$ is denoted by  {\sffamily y1[i,t]}, $\f_i(z)$ is denoted by {\sffamily f[i,z]}, and $\bg_i(z)$ is denoted by {\sffamily g[i,z]}.
The order of perturbation {\sffamily k} must be specified in order to run the code.

\begin{lstlisting}[language=Mathematica,caption={Mathematica's algorithm for computing $\f_i$ \vspace{0.05cm}},mathescape=true]
y0[1, t_] = Integrate[F[1, s, z], {s, 0, t}];
Y0[1] = {y0[1, t]};
For[i = 2, i <= k, i++,
 y0[i, t_] := Integrate[i! F[i, s, z] + Sum[Sum[i!/j! D[F[i - j, t, z], {z, m}] BellY[j, m, Y0[j - m + 1]], {m, 1, j}], {j, 1, i - 1}], {s, 0, t}];
 Y0[i] = Join[Y0[i - 1], {y0[i, t]}];
 f[i, z_] = y0[i, T]/i!];
\end{lstlisting}

\begin{lstlisting}[language=Mathematica,caption={Mathematica's algorithm for computing $\bg_i$ \vspace{0.05cm}},mathescape=true]
g[1, z_] = f[1, z]/T;
y1[1, t] = t g[1, z];
Y1[1, t_] = {y1[1, t]};
For[i = 2, i <= k, i++,
 g[i, z_] = 1/T (f[i, z] -Sum[Sum[1/j! D[g[i - j,z], {z, m}] Integrate[BellY[j, m, Y1[j - m + 1, s]], {s, 0, T}], {m, 1, j}], {j, 1, i - 1}]);
 y1[i, t_] =i! t g[i,z] + Sum[Sum[i!/j! D[g[i - j,z], {z, m}] Integrate[BellY[j, m, Y1[j - m + 1, s]], {s, 0, T}], {m, 1, j}], {j, 1, i - 1}];
 Y1[i, t_] = Join[Y1[i - 1, t], {y1[i, t]}]];
\end{lstlisting}

\section*{Acknowledgements}
The author thanks the referee for the constructive comments and
suggestions which led to an improved version of the manuscript.

The author was partially supported by S\~{a}o Paulo Research Foundation (FAPESP) grants 2018/16430-8, 2018/ 13481-0, and 2019/10269-3, and by Conselho Nacional de Desenvolvimento Cient\'{i}fico e Tecnol\'{o}gico (CNPq) grants 306649/2018-7 and  438975/2018-9. 

\bibliographystyle{abbrv}
\bibliography{Nov2020}

\end{document}